\documentclass[11pt,reqno]{amsart}
\usepackage{amsmath,amssymb,amsthm}
\usepackage[margin=1.1in]{geometry}
\usepackage{booktabs,array}
\usepackage{microtype}
\usepackage{tikz}
\usetikzlibrary{arrows.meta,positioning}
\usepackage[colorlinks=true,linkcolor=blue!60!black,citecolor=blue!60!black,urlcolor=blue!60!black]{hyperref}

\newtheorem{theorem}{Theorem}[section]
\newtheorem{lemma}[theorem]{Lemma}
\newtheorem{proposition}[theorem]{Proposition}
\newtheorem{corollary}[theorem]{Corollary}

\theoremstyle{definition}
\newtheorem{definition}[theorem]{Definition}
\theoremstyle{remark}
\newtheorem{remark}[theorem]{Remark}

\newcommand{\FF}{\mathcal{F}}
\newcommand{\NN}{\mathcal{N}}
\newcommand{\Z}{\mathbb{Z}}
\newcommand{\Q}{\mathbb{Q}}
\newcommand{\ord}{\operatorname{ord}}
\newcommand{\ORD}{\operatorname{ORD}}
\newcommand{\Tm}{\Theta_{-}}
\DeclareMathOperator{\Ggp}{\Gamma_{1}}

\title[Machine-Guided Recurrence Boundary Theory for Nahm Sums]{Machine-Guided Recurrence Boundary Theory for Nahm Sums}
\author{Ankush Goswami}
\address{School of Mathematical and Statistical Sciences, University of Texas Rio Grande Valley, Edinburg, TX 78541, USA}
\email{ankushgoswami3@gmail.com}
\date{August 2026}

\begin{document}

\begin{abstract}
We develop a proof-carrying method for affine families of positive-definite
Nahm sums. The method has three mathematical steps. First, one universal
coordinate-contiguous relation supplies exact algebraic building blocks for
finite certificates. Second, a tropical face-limit theorem identifies
parameter directions along which a higher-rank Nahm sum simplifies to a
lower-rank theta or theta--hypergeometric boundary value. Third, a
recurrence--boundary principle uses enough independent boundary values to
recover the initial sums of the family.

Machine learning and reinforcement learning are used only to search for
useful structure. Here ``machine learning'' means a finite symbolic search:
candidate recurrences and asymptotic rays are generated, scored using exact
algebraic data, and improved by evolutionary selection. ``Reinforcement
learning'' means a sequential proof search: at each step a $Q$-learning agent
chooses one legal contiguous-cell identity to add to the current residual,
with the goal of reducing that residual to zero. These learned searches do
not certify any theorem. Every successful output is replaced by an exact
symbolic identity that can be checked independently.

As the main application we prove the two product identities stated by Shi
and Wang as Conjecture~3.8 of arXiv:2607.23257 for the Nahm sums dual to
Zagier's twelfth rank-three example. Evolutionary symbolic search finds a
second-order recurrence for a one-parameter rank-three family, and a
formula-level $Q$-learning agent finds a five-cell proof of that recurrence.
A second evolutionary search finds two useful asymptotic rays. These rays
lead to a binary theta series and a unary theta series, and hence to two
linear equations for the two unknown initial sums. Solving this $2\times2$
system reduces the conjecture to two generalized-eta identities, which are
certified by the Robins--Frye--Garvan valence method on $\Ggp(300)$ and
$\Ggp(100)$. Consequently the dual of Zagier's twelfth example is modular,
and every member of the affine family is an explicit
$\Z[q,q^{-1}]$-combination of the two base products. Complete verification
artifacts accompany the paper, including independent implementations of the
main proof checks.
\end{abstract}

\subjclass[2020]{11F03, 11P84, 33D15, 68T05; Secondary 90C33}
\keywords{Nahm sums, modularity, Rogers--Ramanujan functions, theta
functions, generalized eta products, valence formula, tropical limits,
linear complementarity, computer-assisted proof, reinforcement learning}

\maketitle
\setcounter{tocdepth}{1}
\tableofcontents

\section{Introduction}\label{sec:intro}

\subsection{Nahm sums and Nahm's problem}
For a symmetric positive-definite matrix $A\in\Q^{r\times r}$, a vector
$B\in\Q^{r}$, and a scalar $C\in\Q$, the associated \emph{Nahm sum} is
\begin{equation}\label{eq:nahmsum}
  f_{A,B,C}(q)
  \;=\;
  \sum_{n=(n_1,\dots,n_r)\in\Z_{\ge0}^{r}}
  \frac{q^{\frac12 n^{T}An+B^{T}n+C}}
       {(q;q)_{n_1}\cdots(q;q)_{n_r}},
\end{equation}
where $(a;q)_n=\prod_{k=0}^{n-1}(1-aq^{k})$. Nahm's
problem~\cite{Nahm07,Zagier07} asks for which triples $(A,B,C)$ the
series \eqref{eq:nahmsum} is a modular function of $\tau$ (with
$q=e^{2\pi i\tau}$); the problem connects conformal field theory, the
Bloch group, and Rogers--Ramanujan-type identities, and one direction of
Nahm's conjecture was proved by Calegari, Garoufalidis, and
Zagier~\cite{CGZ23}. In rank one exactly seven Nahm sums are modular
\cite{Zagier07}. In ranks two and three Zagier compiled candidate lists
of eleven and twelve triples and proved several; the remaining cases were
settled through the work of several authors, with the rank-three list
completed by Wang~\cite{Wang24rank3}; see \cite{VZ11,Wang24rank2} for
history and rank-two analogues. Zagier further observed
\cite[p.~50]{Zagier07} that modularity appears to propagate along the
duality
\begin{equation}\label{eq:duality}
  \mathcal D:\;(A,B,C)\;\longmapsto\;
  \Bigl(A^{-1},\;A^{-1}B,\;\tfrac12 B^{T}A^{-1}B-\tfrac{r}{24}-C\Bigr).
\end{equation}
Shi and Wang~\cite{ShiWang26} investigated the Nahm sums dual to all
twelve rank-three examples and, combining their results with earlier
work, established the modularity of every dual except those of the ninth
and twelfth examples; for the twelfth they gave two precise conjectural
product identities, obtained from exact $5$-dissections
\cite[Conjecture~3.8]{ShiWang26}.

The central difficulty in this subject is usually not recognizing a
plausible product after computing many coefficients---Shi and Wang had
already done this---but finding a structural route from a coupled
multisum to a finite collection of modular objects. This paper develops
such a route for \emph{affine parameter families} and applies it to prove
the two Shi--Wang conjectures.

\subsection{The recurrence--boundary method}\label{subsec:method}
Directly evaluating a complicated multisum is often difficult. Instead, we
place the target sums inside a one-parameter family
\[
  \FF_t(q)=\NN_{A,m}(B_0+t\,d;\,q)
  \qquad(t\in\Z).
\]
The advantage is that neighboring members of the family may satisfy a simple
recurrence even when no single member is easy to sum directly. We therefore
look for three pieces of information:
\begin{itemize}
\item[\textbf{S1.}] a low-order linear recurrence relating nearby values of $\FF_t$;
\item[\textbf{S2.}] a short exact proof of that recurrence using the universal
  coordinate-contiguous relation;
\item[\textbf{S3.}] enough asymptotic directions to determine the initial
  values of the recurrence from simpler theta-valued limits.
\end{itemize}
Once these three pieces are known, the original product conjectures become a
finite system of modular-function identities and can be proved by valence
bounds.

The machine assistance is deliberately separated from the proof. An
evolutionary search proposes simple recurrence formulas and useful asymptotic
directions. A reinforcement-learning agent then searches for a short
sequence of legal contiguous-cell moves that proves the proposed recurrence.
Finally, exact algebra and modular theory check the resulting identities.
Thus the logical structure is
\[
  \text{ML proposes}\;\longrightarrow\;\text{RL plans}
  \;\longrightarrow\;\text{exact verification}.
\]

\subsection{General contributions}\label{subsec:contributions}
The general theory is developed in Section~\ref{sec:general}. Its main points
are the following.
\begin{itemize}
\item[\textbf{G1.}] We prove a universal contiguous relation for Nahm sums with
  common denominator base $(q^{m};q^{m})_{n_i}$. It gives a collection of
  exact local identities, called \emph{cells}. Any finite combination of
  cells that cancels formally gives a genuine identity of Nahm sums
  (Theorem~\ref{thm:contig} and Proposition~\ref{prop:soundness}).
\item[\textbf{G2.}] We prove a \emph{tropical face-limit theorem}
  (Theorem~\ref{thm:tropical}). It tells us what happens when the parameter
  moves far along a suitable affine ray: some summation variables become
  bilateral, some remain unilateral, and some are forced to a boundary
  value. The original sum then reduces to a lower-rank theta or
  theta--hypergeometric expression.
\item[\textbf{G3.}] We prove a \emph{recurrence--boundary determination
  principle} (Theorem~\ref{thm:determination}). An order-$s$ recurrence,
  together with $s$ independent boundary measurements, determines the first
  $s$ members of the family.
\item[\textbf{G4.}] We reduce the final finite system to generalized-eta
  identities, which can be certified by cusp orders and the valence formula
  (Corollary~\ref{cor:modular-cert}).
\item[\textbf{G5.}] We give a \emph{proof-carrying ML/RL implementation}.
  The learned algorithms help find good candidates and short proof paths,
  but exact symbolic verification---not the learned model---establishes every
  theorem (Sections~\ref{subsec:where-ml} and~\ref{sec:ml}).
\end{itemize}

\subsection{Application: the dual of Zagier's twelfth example}
\label{subsec:main}
For $m\ge 1$ and $0<a<m$, write
\begin{equation}\label{eq:Jdef}
  J_m=(q^{m};q^{m})_\infty,
  \qquad
  J_{a,m}=(q^{a},q^{m-a},q^{m};q^{m})_\infty ,
\end{equation}
with $(a_1,\dots,a_s;q)_\infty=\prod_i (a_i;q)_\infty$. Set
\begin{equation}\label{eq:ABdef}
  A=\begin{pmatrix}3&-4&-3\\-4&7&4\\-3&4&8\end{pmatrix},
  \qquad
  B_t=\begin{pmatrix}\tfrac12+2t\\[2pt] -\tfrac32-t\\[2pt] 2-2t\end{pmatrix}
  \qquad(t\in\Z),
\end{equation}
and define the affine family
\begin{equation}\label{eq:Ftdef}
  \FF_t(q)
  \;=\;
  \sum_{i,j,k\ge0}
  \frac{q^{\frac12(i,j,k)A(i,j,k)^{T}+B_t^{T}(i,j,k)}}
       {(q^{5};q^{5})_i\,(q^{5};q^{5})_j\,(q^{5};q^{5})_k}.
\end{equation}
The matrix $A$ is positive definite with $\det A=25$. The two sums of
Shi--Wang's Conjecture~3.8 are $\FF_0$ and $\FF_1$; here and throughout
we normalize without the additive constant $C$ of \eqref{eq:nahmsum},
since the prefactor $q^{C}$ is immaterial to the identities and to
modularity.

\begin{theorem}[Shi--Wang's two product conjectures]
\label{thm:main}
As identities of formal power series in $q$, equivalently for $|q|<1$,
\begin{align}
\FF_0(q)&=
  2\,\frac{J_{50}^{11}}
          {J_{5,50}^{4}J_{10,50}J_{15,50}^{3}J_{20,50}^{2}J_{25}}
  \;+\;q\,\frac{J_{50}^{11}J_{10,50}}
          {J_{5,50}^{5}J_{15,50}^{2}J_{20,50}^{4}J_{25}}
  \;+\;4q^{2}\,\frac{J_{50}^{12}}
          {J_{5,50}^{3}J_{10,50}^{2}J_{15,50}^{3}J_{20,50}J_{25}^{3}},
  \label{eq:mainF0}\\[4pt]
\FF_1(q)&=
  \frac{J_{20,50}J_{50}^{11}}
       {J_{5,50}^{2}J_{10,50}^{4}J_{15,50}^{5}J_{25}}
  \;+\;2q\,\frac{J_{50}^{11}}
       {J_{5,50}^{3}J_{10,50}^{2}J_{15,50}^{4}J_{20,50}J_{25}}
  \;+\;4q^{4}\,\frac{J_{50}^{12}}
       {J_{5,50}^{3}J_{10,50}J_{15,50}^{3}J_{20,50}^{2}J_{25}^{3}}.
  \label{eq:mainF1}
\end{align}
\end{theorem}

Each summand on the right is, up to a power of $q$, a quotient of
generalized Dedekind eta functions, so Theorem~\ref{thm:main} resolves
\cite[Conjecture~3.8]{ShiWang26} and yields:

\begin{corollary}\label{cor:modular}
The Nahm sums dual to Zagier's twelfth rank-three example are modular.
Moreover, the recurrence of Theorem~\ref{thm:rec} propagates modularity
through the entire affine family: for every $t\in\Z$,
\[
  \FF_t\;\in\;\Z[q,q^{-1}]\,\FF_0+\Z[q,q^{-1}]\,\FF_1 ,
\]
with coefficients computable by a two-term matrix iteration
(Section~\ref{sec:consequences}); for instance $\FF_2=\FF_0-q\FF_1$.
\end{corollary}

The first coefficients are
\[
\FF_0=2+q+4q^{2}+8q^{5}+5q^{6}+12q^{7}+22q^{10}+14q^{11}+32q^{12}+\cdots,
\quad
\FF_1=1+2q+4q^{4}+2q^{5}+6q^{6}+12q^{9}+\cdots ;
\]
the supports (exponents $\equiv 0,1,2 \pmod 5$ for $\FF_0$ and
$\equiv0,1,4\pmod 5$ for $\FF_1$) reflect the three-term $5$-dissections
on the product side.

The proof follows the general method above. The machine first suggests a
recurrence, and the paper then proves that recurrence exactly with five
contiguous cells. Because the recurrence has order two, every $\FF_t$ is
controlled by two initial values, $\FF_0$ and $\FF_1$. Two machine-selected
asymptotic rays give two independent equations for these two unknowns. The
corresponding limits are the binary theta series $\Tm$ and the unary theta
series $\vartheta$. Solving the resulting $2\times2$ system reduces
Theorem~\ref{thm:main} to two generalized-eta identities, which are then
proved by valence bounds.

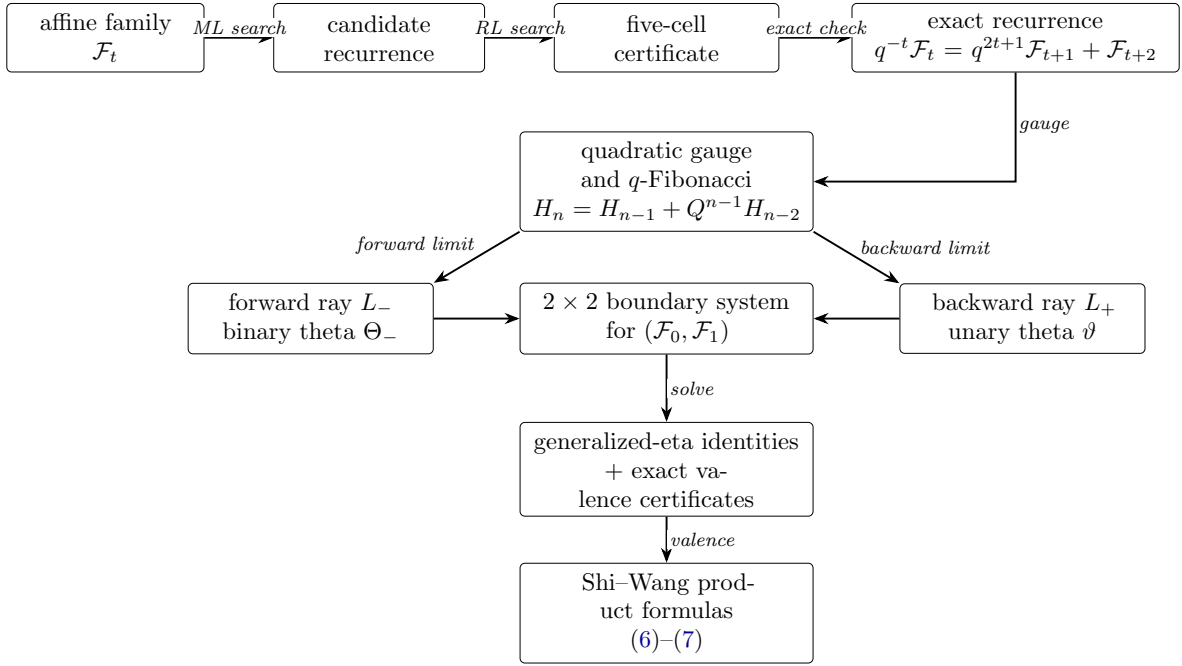
\begin{center}
\resizebox{0.97\textwidth}{!}{%
\begin{tikzpicture}[
  box/.style={
    draw,
    rounded corners=2pt,
    align=center,
    inner sep=4pt,
    font=\small,
    minimum height=10mm
  },
  lab/.style={
    font=\scriptsize\itshape,
    fill=white,
    inner sep=1pt
  },
  arr/.style={
    -{Stealth[length=2.2mm]},
    thick
  }
]


\node[box, text width=2.6cm] (fam)
at (0,0)
{affine family\\ $\FF_t$};

\node[box, text width=2.8cm] (cand)
at (4.0,0)
{candidate\\ recurrence};

\node[box, text width=3.0cm] (cert)
at (8.2,0)
{five-cell\\ certificate};

\node[box, text width=4.5cm] (rec)
at (13.3,0)
{exact recurrence\\
$q^{-t}\FF_t=q^{2t+1}\FF_{t+1}+\FF_{t+2}$};

\draw[arr] (fam) -- node[lab,above]{ML search} (cand);
\draw[arr] (cand) -- node[lab,above]{RL search} (cert);
\draw[arr] (cert) -- node[lab,above]{exact check} (rec);


\node[box, text width=4.0cm] (fib)
at (8.2,-2.1)
{quadratic gauge and $q$-Fibonacci\\
$H_n=H_{n-1}+Q^{n-1}H_{n-2}$};

\node[box, text width=3.3cm] (fwd)
at (3.0,-4.1)
{forward ray $L_-$\\
binary theta $\Tm$};

\node[box, text width=3.3cm] (bwd)
at (13.4,-4.1)
{backward ray $L_+$\\
unary theta $\vartheta$};

\draw[arr] (rec.south) |- node[lab,pos=0.25,right]{gauge} (fib.east);

\draw[arr] (fib.south west) -- node[lab,above left]{forward limit} (fwd.north east);
\draw[arr] (fib.south east) -- node[lab,above right]{backward limit} (bwd.north west);


\node[box, text width=4.0cm] (sys)
at (8.2,-4.1)
{$2\times2$ boundary system\\
for $(\FF_0,\FF_1)$};

\node[box, text width=4.0cm] (val)
at (8.2,-6.3)
{generalized-eta identities\\
+ exact valence certificates};

\node[box, text width=4.0cm] (prod)
at (8.2,-8.4)
{Shi--Wang product formulas\\
\eqref{eq:mainF0}--\eqref{eq:mainF1}};

\draw[arr] (fwd.east) -- (sys.west);
\draw[arr] (bwd.west) -- (sys.east);

\draw[arr] (sys) -- node[lab,right]{solve} (val);
\draw[arr] (val) -- node[lab,right]{valence} (prod);

\end{tikzpicture}%
}
\end{center}

\subsection{What the machine does and what is proved exactly}
\label{subsec:epistemics}
Machine learning is used for discovery, not for mathematical certification.
The evolutionary searches suggest a recurrence and two asymptotic rays. A
formula-level $Q$-learning agent then searches for a short sequence of legal
contiguous-cell identities that proves the recurrence. After these objects
are found, the learned searches are no longer needed: the recurrence is
verified as an exact identity in a free module over
$\Z[q^{\pm1},q^{\pm t}]$, the two boundary limits follow from
Theorem~\ref{thm:tropical}, and the final product identities follow from
exact cusp-order arithmetic and the valence formula. No floating-point
value of $q$ is used in the proof. Section~\ref{sec:ml} explains the search
algorithms in elementary mathematical terms, and
Section~\ref{sec:verification} describes the independent verification
artifacts.

\subsection{Outline}
Section~\ref{sec:prelim} fixes notation and the series topology.
Section~\ref{sec:general} develops the general recurrence--boundary
framework (G1--G4). Sections~\ref{sec:recurrence}--\ref{sec:certificates}
carry out the application: the recurrence and its certificate
(Section~\ref{sec:recurrence}), the quadratic gauge and two-sided
dynamics (Section~\ref{sec:gauge}), the two tropical rays
(Sections~\ref{sec:forward}--\ref{sec:backward}), the boundary system and
the reduction of Theorem~\ref{thm:main}
(Section~\ref{sec:system}), and the valence certificates
(Section~\ref{sec:certificates}). Section~\ref{sec:ml} describes the
machine-learning implementation, Section~\ref{sec:verification} the
verification artifacts, Section~\ref{sec:consequences} consequences and
the remaining Shi--Wang targets. Appendix~\ref{app:exponents} collects the
exact exponent identities, Appendix~\ref{app:base} the common-base
conversion, and Appendix~\ref{app:artifacts} the artifact inventory.

A reader interested only in the application can read
Sections~\ref{sec:prelim} and~\ref{sec:recurrence}--\ref{sec:certificates},
taking the two boundary evaluations from the self-contained direct
arguments given there; a reader interested in the general method should
start with Section~\ref{sec:general}.

\section{Formal series, topology, and auxiliary products}
\label{sec:prelim}

\subsection{Series and coefficientwise convergence}\label{subsec:topology}
Fix $M\ge1$ and work in the field
$K=\Q((q^{1/M}))$ of formal Puiseux--Laurent series with bounded
denominators; in the application $M=1$ and all series lie in
$\Z((q))\otimes\Q$. We write $[q^{x}]f$ for the coefficient of $q^{x}$.

\begin{definition}[Coefficientwise convergence]\label{def:topology}
A sequence $(f_h)_{h\ge0}$ in $K$ \emph{converges coefficientwise} to
$f\in K$ if (i) the orders $\ord_q f_h$ are bounded below uniformly in
$h$, and (ii) for every exponent $x$ the coefficient $[q^{x}]f_h$ is
eventually constant, equal to $[q^{x}]f$.
\end{definition}

With this notion, sums and products of convergent sequences converge to
the corresponding sums and products: for products, a fixed coefficient of
$f_hg_h$ involves only the finitely many coefficients of $f_h$ and $g_h$
in a window determined by the uniform lower bounds, each of which is
eventually constant. All limits in this paper are coefficientwise limits
in this sense.

Positive definiteness enters through the following standard facts, which
we use silently: a positive-definite quadratic $\tfrac12u^{T}Au$ plus a
linear form is bounded below on $\Z^{r}$ and proper, so its sublevel sets
are finite; consequently every Nahm-type sum below has a finite principal
part and well-defined coefficients.

\begin{lemma}[Stabilization]\label{lem:stab}
Let $Q=q^{m}$. For every $n\ge0$,
\[
  \frac{1}{(Q;Q)_n}-\frac{1}{(Q;Q)_\infty}\;\in\;Q^{n+1}\,\Z[[Q]],
\]
i.e.\ the two series agree through $q^{m(n+1)-1}$. Likewise, for
$0\le 2r\le n$, the Gaussian binomial
$\genfrac{[}{]}{0pt}{1}{n-r}{r}_Q$
agrees with $\frac{1}{(Q;Q)_r}$ through $Q^{\,n-2r}$.
\end{lemma}

\begin{proof}
$\frac{1}{(Q;Q)_n}=\frac{(Q^{n+1};Q)_\infty}{(Q;Q)_\infty}$ with
$(Q^{n+1};Q)_\infty\in 1+Q^{n+1}\Z[[Q]]$, and
$\genfrac{[}{]}{0pt}{1}{n-r}{r}_Q
=\frac{(Q^{n-2r+1};Q)_{r}}{(Q;Q)_{r}}$
with $(Q^{n-2r+1};Q)_{r}\in 1+Q^{\,n-2r+1}\Z[[Q]]$.
\end{proof}

\subsection{Auxiliary products for the application}
\label{subsec:products}
Throughout the application,
\begin{equation}\label{eq:Qdef}
  Q=q^{5},
  \qquad
  j(x;R)=(x,\,R/x,\,R;\,R)_\infty .
\end{equation}
We use the Rogers--Ramanujan functions, in the variable $Q$,
\begin{equation}\label{eq:GH}
  G(Q)=\sum_{r\ge0}\frac{Q^{r^{2}}}{(Q;Q)_r}
      =\frac{1}{(Q,Q^{4};Q^{5})_\infty},
  \qquad
  H(Q)=\sum_{r\ge0}\frac{Q^{r^{2}+r}}{(Q;Q)_r}
      =\frac{1}{(Q^{2},Q^{3};Q^{5})_\infty},
\end{equation}
the product forms being the Rogers--Ramanujan identities
\cite{Andrews98}, together with
\begin{equation}\label{eq:OmDS}
  \Omega(Q)=(-Q;Q)_\infty,
  \qquad
  D(Q)=\Omega(Q)\,G(Q^{4}),
  \qquad
  S(Q)=\Omega(Q)\,H(Q^{4}),
\end{equation}
so that, in the $J$-notation \eqref{eq:Jdef} with base variable $q$,
\[
  G(Q)=\frac{J_{25}}{J_{5,25}},\quad
  H(Q)=\frac{J_{25}}{J_{10,25}},\quad
  \Omega(Q)=\frac{J_{10}}{J_{5}},\quad
  D(Q)=\frac{J_{10}J_{100}}{J_{5}J_{20,100}},\quad
  S(Q)=\frac{J_{10}J_{100}}{J_{5}J_{40,100}}.
\]
We also use the unary theta function
\begin{equation}\label{eq:vartheta}
  \vartheta(q)=\sum_{n\in\Z}q^{n^{2}}=\frac{J_2^{5}}{J_1^{2}J_4^{2}}
\end{equation}
and the binary theta series
\begin{equation}\label{eq:ThetaM}
  \Tm(q)=\sum_{R,S\in\Z}q^{(3R^{2}-8RS+R+7S^{2}-3S)/2},
\end{equation}
whose quadratic part has Gram matrix
$\bigl(\begin{smallmatrix}3&-4\\-4&7\end{smallmatrix}\bigr)$ of
determinant $5$; the exponent in \eqref{eq:ThetaM} is a nonnegative
integer for all $(R,S)\in\Z^{2}$, so $\Tm\in\Z[[q]]$.

\section{The general recurrence--boundary framework}
\label{sec:general}

\subsection{Affine Nahm sums}\label{subsec:affine}
Fix a positive integer $m$ and write $Q=q^{m}$. Let $A$ be a symmetric
positive-definite rational $r\times r$ matrix. For $B\in\Q^{r}$ define
\begin{equation}\label{eq:NAm}
  \NN_{A,m}(B;q)
  =\sum_{n\in\Z_{\ge0}^{r}}
  \frac{q^{\frac12 n^{T}An+B^{T}n}}
       {\prod_{j=1}^{r}(Q;Q)_{n_j}}
  \;\in\;K .
\end{equation}
Positive definiteness gives coefficientwise finiteness and a finite
principal part, as in Section~\ref{subsec:topology}.

\subsection{The contiguous-cell module}\label{subsec:cells}
Let $e_j$ denote the $j$-th coordinate vector and $Ae_j$ the $j$-th
column of $A$.

\begin{theorem}[Universal coordinate-contiguous relation]
\label{thm:contig}
For every $B\in\Q^{r}$ and $1\le j\le r$,
\begin{equation}\label{eq:contig}
  \NN_{A,m}(B)-\NN_{A,m}(B+me_j)
  \;=\;q^{\,B_j+A_{jj}/2}\,\NN_{A,m}(B+Ae_j).
\end{equation}
\end{theorem}

\begin{proof}
Since $(B+me_j)^{T}n=B^{T}n+mn_j$, the $n$-th summand of the difference
is the $n$-th summand of $\NN_{A,m}(B)$ multiplied by
$1-q^{mn_j}=1-Q^{n_j}$. Using
$(1-Q^{n_j})/(Q;Q)_{n_j}=1/(Q;Q)_{n_j-1}$ for $n_j\ge1$, the terms with
$n_j=0$ vanish; shifting $n_j\mapsto n_j+1$ in the remaining sum
increases the exponent by $B_j+A_{jj}/2+(Ae_j)^{T}n$ and produces the
right side. At any fixed power of $q$ only finitely many lattice points
appear on either side, so the manipulation is legitimate in $K$.
\end{proof}

Introduce formal symbols $[B]$ indexed by parameter vectors and define
the \emph{cell}
\begin{equation}\label{eq:cell-general}
  C_j(B)\;=\;[B]-[B+me_j]-q^{\,B_j+A_{jj}/2}\,[B+Ae_j].
\end{equation}
Let $\mathcal C_{A,m}$ be the submodule generated by all cells inside the
free module on the symbols $[B]$ over the ring of Laurent polynomials in
the relevant powers of $q$.

\begin{proposition}[Certificate soundness]\label{prop:soundness}
If a finite Laurent-polynomial combination of cells equals a formal
target
\[
  R=\sum_{\nu=1}^{N}c_\nu(q)\,[B_\nu]
\]
in the free module, then
$\sum_{\nu=1}^{N}c_\nu(q)\,\NN_{A,m}(B_\nu;q)=0$. In particular, a
learned policy can affect \emph{which} certificate is found, but not the
validity of a certificate that passes exact free-module verification.
\end{proposition}

\begin{proof}
Apply the evaluation homomorphism $[B]\mapsto\NN_{A,m}(B;q)$. Every
cell maps to zero by Theorem~\ref{thm:contig}, hence so does every
finite combination of cells.
\end{proof}

\subsection{Tropical face limits}\label{subsec:tropical}
The next theorem is the principal general result. It describes how
affine motion of the linear parameter forces a Nahm sum onto a
lower-dimensional face, and which face.

\begin{theorem}[Tropical face-limit theorem]\label{thm:tropical}
Let $\beta,\delta\in\Q^{r}$ and consider the ray
\[
  F_h(q)=\NN_{A,m}(\beta+h\delta;\,q),
  \qquad h\in\Z_{\ge0}.
\]
Suppose $v\in\Z_{\ge0}^{r}$ satisfies the \emph{linear complementarity
condition}
\begin{equation}\tag{TF}\label{eq:TF}
  w:=Av+\delta\in\Q_{\ge0}^{r},
  \qquad
  v_jw_j=0
  \quad(1\le j\le r).
\end{equation}
Partition the coordinates as
\[
  I=\{j:v_j>0\},
  \qquad
  J=\{j:v_j=0,\;w_j=0\},
  \qquad
  \mathcal K=\{j:w_j>0\}.
\]
Then, coefficientwise,
\begin{equation}\label{eq:tropical-limit}
  q^{\frac12h^{2}v^{T}Av-h\beta^{T}v}\;F_h(q)
  \;\longrightarrow\;
  \frac{1}{(Q;Q)_\infty^{\,|I|}}
  \sum_{\substack{u_i\in\Z\ (i\in I)\\ u_j\in\Z_{\ge0}\ (j\in J)\\
                  u_k=0\ (k\in\mathcal K)}}
  \frac{q^{\frac12u^{T}Au+\beta^{T}u}}
       {\prod_{j\in J}(Q;Q)_{u_j}} .
\end{equation}
\end{theorem}

\begin{proof}
Substitute $n=hv+u$, so that $u$ ranges over the shifted cone
$u_j\ge-hv_j$. Expanding and adding the normalizer,
\begin{align*}
  &\tfrac12(hv+u)^{T}A(hv+u)+(\beta+h\delta)^{T}(hv+u)
   +\tfrac12h^{2}v^{T}Av-h\beta^{T}v\\
  &\qquad=\;h^{2}\,v^{T}(Av+\delta)\;+\;h\,u^{T}(Av+\delta)
   \;+\;\tfrac12u^{T}Au+\beta^{T}u
  \;=\;h\,w^{T}u+\tfrac12u^{T}Au+\beta^{T}u,
\end{align*}
where the $h^{2}$-term vanishes because $v^{T}w=0$ by \eqref{eq:TF}
(each product $v_jw_j$ is zero). Now fix an exponent $x$ and compute
$[q^{x}]$ of the left side of \eqref{eq:tropical-limit} for large $h$.
The residual exponent
$\rho(u)=\tfrac12u^{T}Au+\beta^{T}u$ is bounded below on $\Z^{r}$, say
by $-C_0$, and proper.
\begin{itemize}
\item For $k\in\mathcal K$ we have $v_k=0$, so $u_k=n_k\ge0$, and any
  term with $u_k\ge1$ has exponent at least $h\,w_k-C_0$, hence
  contributes nothing to $[q^{x}]$ once $h$ is large. This forces
  $u_k=0$.
\item On the face $u_k=0$ $(k\in\mathcal K)$ the exponent is
  $\rho(u)$, independent of $h$; by properness only a finite set
  $\mathcal L_x$ of lattice points $u$ has $\rho(u)\le x$, independent
  of $h$.
\item For $i\in I$, complementarity gives $w_i=0$; for $h$ large every
  $u\in\mathcal L_x$ satisfies $u_i\ge-hv_i$, so the constraint is
  inactive and $u_i$ is effectively bilateral, while by
  Lemma~\ref{lem:stab} the factor $(Q;Q)_{hv_i+u_i}^{-1}$ agrees with
  $(Q;Q)_\infty^{-1}$ through $q^{x}$ once $m(hv_i+u_i+1)>x$.
\item For $j\in J$ we have $n_j=u_j\ge0$, and the factor
  $(Q;Q)_{u_j}^{-1}$ is untouched.
\end{itemize}
Hence for all large $h$ the coefficient $[q^{x}]$ of the left side
equals $[q^{x}]$ of the right side of \eqref{eq:tropical-limit}; the
orders are bounded below by $-C_0$ uniformly, so the convergence is
coefficientwise in the sense of Definition~\ref{def:topology}.
\end{proof}

\begin{remark}[Linear complementarity geometry]\label{rem:LCP}
Condition \eqref{eq:TF} is a linear complementarity problem:
\[
  v\ge0,\qquad Av+\delta\ge0,\qquad v_j(Av+\delta)_j=0 .
\]
The support of $v$ predicts the bilateral variables (each contributing a
factor $(Q;Q)_\infty^{-1}$), the zero--zero coordinates predict residual
unilateral $q$-hypergeometric variables, and the strictly positive
coordinates of $w$ predict variables forced onto the boundary face. The
finite search for admissible integral $v$ (and for the ray data below)
is exactly the kind of bounded combinatorial exploration at which the
evolutionary methods of Section~\ref{sec:ml} excel; an accepted solution
is then verified by the two displayed linear computations.
\end{remark}

\begin{remark}[Arithmetic subfamilies]\label{rem:subfamily}
It may happen that no admissible $v$ exists for the step direction
$\delta$ itself but one exists for a multiple $a\delta$, $a\ge2$. Since
$\NN_{A,m}(\beta+(ah+b)\delta)$ is again a ray with data
$(\beta+b\delta,\;a\delta)$, Theorem~\ref{thm:tropical} then applies
along the arithmetic progression $t=ah+b$. Both phenomena occur in the
application: the forward ray uses $a=1$ and the backward ray requires
$a=4$ (Section~\ref{sec:backward}).
\end{remark}

\subsection{Recurrence--boundary determination}\label{subsec:determination}

\begin{theorem}[Recurrence--boundary determination principle]
\label{thm:determination}
Let $(F_t)_{t\in\Z}$ in $K$ satisfy an order-$s$ linear recurrence whose
coefficients are units of $K$, and let $c_{t,j}\in K$ $(0\le j<s)$ be the
fundamental solutions of the same recurrence, characterized by
$c_{j',j}=\delta_{j'j}$ for $0\le j'<s$; then $F_t=\sum_{j<s}c_{t,j}F_j$
for all $t$. Assume that for $1\le\nu\le s$ there are integer sequences
$t_\nu(h)$ and gauges $\gamma_\nu(h)\in K$ such that, coefficientwise,
\[
  \gamma_\nu(h)\,F_{t_\nu(h)}\longrightarrow\Theta_\nu,
  \qquad
  \gamma_\nu(h)\,c_{t_\nu(h),\,j}\longrightarrow\ell_{\nu j}
  \quad(0\le j<s).
\]
If the matrix $L=(\ell_{\nu j})$ is invertible over $K$, then
\begin{equation}\label{eq:determination}
  (F_0,\dots,F_{s-1})^{T}=L^{-1}(\Theta_1,\dots,\Theta_s)^{T}.
\end{equation}
Thus the initial sums are determined by the recurrence together with $s$
independent boundary measurements.
\end{theorem}

\begin{proof}
Transporting $F_{t_\nu(h)}$ to the basis window and multiplying by the
gauge, $\gamma_\nu(h)F_{t_\nu(h)}=\sum_{j<s}\gamma_\nu(h)
c_{t_\nu(h),j}\,F_j$. By the continuity of sums and products under
coefficientwise convergence (Section~\ref{subsec:topology}), the right
side converges to $\sum_j\ell_{\nu j}F_j$, giving
$\Theta_\nu=\sum_j\ell_{\nu j}F_j$ for each $\nu$; invertibility of $L$
yields \eqref{eq:determination}.
\end{proof}

\begin{corollary}[Finite modular certification]\label{cor:modular-cert}
In the setting of Theorem~\ref{thm:determination}, suppose the entries of
$L$, the boundary values $\Theta_\nu$, and proposed closed forms
$P_0,\dots,P_{s-1}$ for $F_0,\dots,F_{s-1}$ are finite combinations of
generalized-eta quotients on a common congruence subgroup. Then the
statements $F_j=P_j$ reduce to the $s$ modular-function identities
$\sum_j\ell_{\nu j}P_j=\Theta_\nu$; each is decidable by an exact
non-infinity cusp-order bound together with a $q$-expansion beyond the
corresponding valence bound (Section~\ref{sec:certificates}).
\end{corollary}

\begin{proof}
If the $P_j$ satisfy the same invertible system as the $F_j$, then
$P_j=F_j$ by uniqueness of the solution; the valence criterion decides
each of the $s$ product identities as in
Section~\ref{subsec:valence}.
\end{proof}

\subsection{Where machine learning and reinforcement learning enter}
\label{subsec:where-ml}
It is useful to separate two questions:
\emph{How is a promising object found?} and \emph{How is it proved?}
The learned algorithms answer only the first question. Once a candidate has
been found, an exact argument replaces the learned search.

\begin{center}
\begin{tabular}{p{3.5cm}p{5.1cm}p{5.1cm}}
\toprule
\textbf{Stage} & \textbf{What the machine searches for} & \textbf{How it is proved exactly}\\
\midrule
Affine discovery & a direction $d$, a small recurrence order, and simple monomial coefficients & exhaustive residual sweep; certificate of Theorem~\ref{thm:rec}\\\mbox{}\\
Proof synthesis & a short sequence of shifted cells and monomial multipliers & free-module cancellation (Proposition~\ref{prop:soundness})\\\mbox{}\\
Boundary discovery & a complementarity velocity $v$ and ray data $(a,b)$ & Theorem~\ref{thm:tropical} and direct exponent calculations\\\mbox{}\\
Product recognition & short generalized-eta expressions & cusp orders, valence formula, exact integer coefficients\\
\bottomrule
\end{tabular}
\end{center}

For example, the recurrence search is not asked to invent an arbitrary
formula. It searches a small, explicitly specified family of formulas and
compares their exact Laurent coefficients. Similarly, the RL agent is not
allowed to write arbitrary proof steps: every action is one instance of the
already-proved contiguous relation. Thus learning changes the efficiency of
the search, but it does not enlarge the set of logically allowed proof
steps.

This distinction is important. Once an affine direction and an ansatz are
fixed, classical packages such as \textsf{qMultiSum}~\cite{Riese03} can
derive recurrences for $q$-hypergeometric multisums. The harder exploratory
question is which directions, recurrence orders, sparse certificates, and
asymptotic rays are worth trying. The learned algorithms are used for that
finite combinatorial exploration; exact symbolic mathematics remains the
sole authority for every accepted statement.

\section{The machine-discovered recurrence and its five-cell certificate}
\label{sec:recurrence}

We now specialize to the data \eqref{eq:ABdef}, with $m=5$, $Q=q^{5}$,
and write $F(B)=\NN_{A,5}(B;q)$, so that $\FF_t=F(B_t)$ and
\begin{equation}\label{eq:Bt-affine}
  B_t=B_0+t\,d,
  \qquad
  B_0=\Bigl(\tfrac12,\,-\tfrac32,\,2\Bigr)^{T},
  \qquad
  d=(2,-1,-2)^{T}.
\end{equation}
For $v\in\Z^{3}$ write $[v]:=F(B_t+v)$, so $[0]=\FF_t$,
$[d]=\FF_{t+1}$, $[2d]=\FF_{t+2}$. Applying Theorem~\ref{thm:contig}
with $B=B_t+v$ and $(B_t)_r=(B_0)_r+t\,d_r$ gives, for each coordinate
$r$ and shift $v$, the \emph{vanishing cell}
\begin{equation}\label{eq:cell}
  C_r(v)\;:=\;[v]-[v+5e_r]-q^{\,d_r t+v_r+c_r}\,[v+Ae_r]
  \;=\;0,
  \qquad
  c:=B_0+\tfrac12\operatorname{diag}(A)=(2,2,6)^{T},
\end{equation}
since $d_rt+v_r+c_r=(B_t+v)_r+A_{rr}/2$.

\begin{theorem}[The recurrence]\label{thm:rec}
For every $t\in\Z$, in $\Z((q))$,
\begin{equation}\label{eq:rec}
  q^{-t}\,\FF_t\;=\;q^{2t+1}\,\FF_{t+1}+\FF_{t+2}.
\end{equation}
\end{theorem}

\begin{proof}
Set $R_t:=q^{-t}[0]-q^{2t+1}[d]-[2d]$, with $d=(2,-1,-2)$,
$2d=(4,-2,-4)$. We claim the exact identity of formal symbols
\begin{equation}\label{eq:certificate}
  R_t\;=\;-\,C_2(2d)\;+\;C_1(-1,3,1)\;+\;q^{2-2t}\,C_1(-4,7,4)
        \;-\;C_3(4,3,-4)\;+\;q^{-t}\,C_2(0).
\end{equation}
Expanding each cell by \eqref{eq:cell}, with the columns
$Ae_1=(3,-4,-3)$, $Ae_2=(-4,7,4)$, $Ae_3=(-3,4,8)$:
\begin{align*}
C_2(4,-2,-4)&=[4,-2,-4]-[4,3,-4]-q^{-t}\,[0,5,0],\\
C_1(-1,3,1)&=[-1,3,1]-[4,3,1]-q^{2t+1}\,[2,-1,-2],\\
C_1(-4,7,4)&=[-4,7,4]-[1,7,4]-q^{2t-2}\,[-1,3,1],\\
C_3(4,3,-4)&=[4,3,-4]-[4,3,1]-q^{-2t+2}\,[1,7,4],\\
C_2(0,0,0)&=[0,0,0]-[0,5,0]-q^{2-t}\,[-4,7,4].
\end{align*}
Substituting into the right side of \eqref{eq:certificate}, the
auxiliary symbols cancel in pairs: $[4,3,-4]$ between the first and
fourth lines; $q^{-t}[0,5,0]$ between the first and fifth; $[-1,3,1]$
between the second and third (after the prefactor $q^{2-2t}$);
$[4,3,1]$ between the second and fourth; and $q^{2-2t}[-4,7,4]$ and
$q^{2-2t}[1,7,4]$ among the third, fourth, and fifth. What survives is
exactly $q^{-t}[0]-q^{2t+1}[2,-1,-2]-[4,-2,-4]=R_t$, proving
\eqref{eq:certificate}. By Proposition~\ref{prop:soundness} the
evaluation of $R_t$ vanishes, which is \eqref{eq:rec}. The accompanying
verifiers additionally expand \eqref{eq:certificate} in the free module
over $\Z[q^{\pm1},q^{\pm t}]$ and confirm every coefficient
mechanically.
\end{proof}

\begin{remark}[Laurent caveat]\label{rem:laurent}
Identity \eqref{eq:rec} holds in $\Z((q))$: $\FF_t\in\Z[[q]]$ for
$t=0,1,2$, but general members of the family have finite principal
parts---for instance $\FF_3$ contains the term $q^{-1}$ from
$(i,j,k)=(0,1,0)$. For $t=0,1$ the recurrence specializes to
\begin{equation}\label{eq:rec01}
  \FF_0=q\FF_1+\FF_2,
  \qquad
  \FF_1=q^{4}\FF_2+q\FF_3,
\end{equation}
the second already involving the principal part of $\FF_3$; both were
re-verified coefficientwise through $q^{120}$ by direct expansion.
\end{remark}

\begin{remark}[How the recurrence and its proof were found]\label{rem:provenance}
The recurrence \eqref{eq:rec} was not given to the search program in advance.
The evolutionary search tested the finite family
\[
  \FF_{t+2}=\varepsilon_0q^{a_0t+b_0}\FF_t
  +\varepsilon_1q^{a_1t+b_1}\FF_{t+1},
\]
with bounded integer parameters. Among the $54{,}756$ candidates, the tuple
$(-1,0,+1;\,2,1,-1)$ was the unique candidate whose exact Laurent
coefficients vanished throughout the search window; an exhaustive sweep of
the same grammar confirmed uniqueness. This suggested \eqref{eq:rec}.

The proof was found separately. A formula-level $Q$-learning agent was given
only the legal cell identity \eqref{eq:cell}. Starting from the residual
$R_t$, it repeatedly chose one shifted cell and one Laurent-monomial
multiplier, trying to reach the zero residual. It found the five-cell
certificate \eqref{eq:certificate}. A breadth-first search of $298{,}442$
states through depth four found no shorter certificate in the released
search environment. The displayed five-cell identity, not the learning
process, is the proof.
\end{remark}

\section{Quadratic gauge and two-sided \texorpdfstring{$q$}{q}-Fibonacci dynamics}
\label{sec:gauge}

\begin{lemma}[Gauge transformation]\label{lem:gauge}
Define
\begin{equation}\label{eq:gauge}
  G_t=q^{\,t(3t-1)/2}\,\FF_t
  \quad(t\in\Z),
  \qquad
  H_n=G_{-n}.
\end{equation}
Then \eqref{eq:rec} is equivalent to
\begin{equation}\label{eq:fib}
  H_n=H_{n-1}+Q^{\,n-1}H_{n-2}
  \qquad(n\in\Z),
\end{equation}
with
\begin{equation}\label{eq:init}
  H_0=\FF_0,
  \qquad
  H_1=\FF_0+q\,\FF_1 .
\end{equation}
\end{lemma}

\begin{proof}
Write $g(t)=t(3t-1)/2$, so $\FF_t=q^{-g(t)}G_t$. Substituting into
\eqref{eq:rec} and comparing exponents,
\[
  -t-g(t)=2t+1-g(t+1),
  \qquad
  -g(t+2)=-t-g(t)-5(t+1),
\]
both elementary polynomial identities in $t$. Dividing \eqref{eq:rec} by
$q^{-t-g(t)}$ gives $G_t=G_{t+1}+Q^{-(t+1)}G_{t+2}$, and $n=-t$ yields
\eqref{eq:fib}. For \eqref{eq:init}: $H_0=G_0=\FF_0$, while
$H_1=G_{-1}=q^{2}\FF_{-1}$ and the case $t=-1$ of \eqref{eq:rec} reads
$q\FF_{-1}=q^{-1}\FF_0+\FF_1$, i.e.\ $q^{2}\FF_{-1}=\FF_0+q\FF_1$.
\end{proof}

\subsection{Forward direction: Schur polynomials}
Let $(\mathcal C_n)$ and $(\mathcal B_n)$ solve \eqref{eq:fib} with
$\mathcal C_0=\mathcal C_1=1$ and $\mathcal B_0=0$, $\mathcal B_1=1$;
these are the fundamental solutions of
Theorem~\ref{thm:determination} in the gauged normalization.

\begin{lemma}[Schur forms]\label{lem:schur}
For all $n\ge0$,
\begin{equation}\label{eq:schur}
  \mathcal C_n=\sum_{r\ge0}Q^{r^{2}}
  \genfrac{[}{]}{0pt}{0}{n-r}{r}_{Q},
  \qquad
  \mathcal B_n=\sum_{r\ge0}Q^{r^{2}+r}
  \genfrac{[}{]}{0pt}{0}{n-r-1}{r}_{Q}.
\end{equation}
\end{lemma}

\begin{proof}
Both hold for $n=0,1$. The Gaussian--Pascal rule
$\genfrac{[}{]}{0pt}{1}{m}{r}_Q
 =\genfrac{[}{]}{0pt}{1}{m-1}{r}_Q
 +Q^{\,m-r}\genfrac{[}{]}{0pt}{1}{m-1}{r-1}_Q$
shows that the right sides satisfy \eqref{eq:fib}; induction completes
the proof. These are the classical Schur polynomials attached to the
Rogers--Ramanujan continued fraction \cite{Andrews98}.
\end{proof}

By linearity, $H_n=\mathcal C_n\FF_0+q\,\mathcal B_n\FF_1$ for $n\ge0$.
Letting $n\to\infty$ and using Lemma~\ref{lem:stab} (for a fixed power
of $Q$ only finitely many $r$ contribute, and each Gaussian binomial has
stabilized once $n-2r$ exceeds it),
\begin{equation}\label{eq:fwd-lin}
  \gamma_-(n)\,c\text{-limits:}\qquad
  \mathcal C_n\to G(Q),\quad \mathcal B_n\to H(Q),
  \qquad\text{so}\qquad
  \lim_{n\to\infty}H_n=G(Q)\,\FF_0+q\,H(Q)\,\FF_1 .
\end{equation}

\subsection{Backward direction: continuants}
For $m\ge0$ define
\begin{equation}\label{eq:EOdef}
  E_m=Q^{-m^{2}}H_{-2m},
  \qquad
  O_m=Q^{-m(m+1)}H_{-(2m+1)} .
\end{equation}
Running \eqref{eq:fib} backwards, $H_{n-2}=Q^{1-n}(H_n-H_{n-1})$, and an
exponent count gives
\begin{equation}\label{eq:EOrec}
  E_{m+1}=E_m-Q^{m}\,O_m,
  \qquad
  O_{m+1}=O_m-Q^{m+1}\,E_{m+1},
\end{equation}
with $E_0=\FF_0$ and $O_0=H_{-1}=G_1=q\,\FF_1$.

\begin{lemma}[Finite backward continuants]\label{lem:continuant}
For all $m\ge0$,
\begin{equation}\label{eq:Om}
  O_m=
  -\,Q\sum_{r=0}^{m-1}Q^{\,r^{2}+2r}
     \genfrac{[}{]}{0pt}{0}{m+r}{2r+1}_{Q}\,E_0
  \;+\;
  \sum_{r=0}^{m}Q^{\,r^{2}}
     \genfrac{[}{]}{0pt}{0}{m+r}{2r}_{Q}\,O_0 .
\end{equation}
\end{lemma}

\begin{proof}
Write $O_m=\alpha_mE_0+\beta_mO_0$ and
$E_m=\gamma_mE_0+\delta_mO_0$, with the four recursions induced by
\eqref{eq:EOrec}. The claimed
$\alpha_m=-Q\sum_{r<m}Q^{r^{2}+2r}\genfrac{[}{]}{0pt}{1}{m+r}{2r+1}_Q$
and
$\beta_m=\sum_{r\le m}Q^{r^{2}}\genfrac{[}{]}{0pt}{1}{m+r}{2r}_Q$
hold at $m=0$; the inductive step is the Gaussian--Pascal rule in the
forms
$\genfrac{[}{]}{0pt}{1}{m+1+r}{2r+1}_Q
=\genfrac{[}{]}{0pt}{1}{m+r}{2r+1}_Q
+Q^{\,m-r}\genfrac{[}{]}{0pt}{1}{m+r}{2r}_Q$
and its even-index analogue. The bookkeeping is also checked as an exact
polynomial identity for $m\le20$ by two independent implementations.
\end{proof}

Since each Gaussian binomial stabilizes and only finitely many $r$
contribute per degree,
\begin{equation}\label{eq:bwd-lin}
  \lim_{m\to\infty}O_m
  =-\,Q\Biggl(\sum_{r\ge0}\frac{Q^{r^{2}+2r}}{(Q;Q)_{2r+1}}\Biggr)\FF_0
  +q\Biggl(\sum_{r\ge0}\frac{Q^{r^{2}}}{(Q;Q)_{2r}}\Biggr)\FF_1 .
\end{equation}
The two interior sums are Slater's identities S.~96 and S.~79
\cite{Slater52}:
\begin{align}
  \sum_{r\ge0}\frac{Q^{r^{2}}}{(Q;Q)_{2r}}
  &=\frac{J_2(Q)\,J_{20}(Q)}{J_1(Q)\,J_{4,20}(Q)}
   =\Omega(Q)\,G(Q^{4})=D(Q),
  \label{eq:S79}\\
  \sum_{r\ge0}\frac{Q^{r^{2}+2r}}{(Q;Q)_{2r+1}}
  &=\frac{J_2(Q)\,J_{20}(Q)}{J_1(Q)\,J_{8,20}(Q)}
   =\Omega(Q)\,H(Q^{4})=S(Q),
  \label{eq:S96}
\end{align}
the middle products written in the variable $Q$ (both re-verified
independently through $Q^{60}$). Consequently
\begin{equation}\label{eq:bwd-lin2}
  \lim_{m\to\infty}O_m
  \;=\;-\,Q\,S(Q)\,\FF_0+q\,D(Q)\,\FF_1 .
\end{equation}

\section{The forward tropical ray}
\label{sec:forward}

We now compute $\lim H_n$ a second time, directly from the Nahm sum, as
an instance of Theorem~\ref{thm:tropical}. Since $B_{-N}=B_0+N(-d)$,
the sequence $(\FF_{-N})_{N\ge0}$ is the ray with data
\[
  \beta=B_0,
  \qquad
  \delta=-d=(-2,1,2)^{T}.
\]
The machine-selected velocity is $v=(2,1,0)^{T}$
(Section~\ref{sec:ml}), and one checks in two lines that it solves
\eqref{eq:TF} \emph{exactly}, in the interior regime:
\[
  Av=\begin{pmatrix}3&-4&-3\\-4&7&4\\-3&4&8\end{pmatrix}
     \begin{pmatrix}2\\1\\0\end{pmatrix}
    =\begin{pmatrix}2\\-1\\-2\end{pmatrix}=d,
  \qquad
  w=Av+\delta=d-d=0 ,
\]
so $I=\{1,2\}$, $J=\{3\}$, $\mathcal K=\varnothing$: the first two
coordinates become bilateral and the third remains unilateral.
Moreover the tropical normalizer coincides with the quadratic gauge of
Lemma~\ref{lem:gauge}:
\[
  \tfrac12 v^{T}Av=\tfrac12 v^{T}d=\tfrac32,
  \qquad
  \beta^{T}v=-\tfrac12,
  \qquad\text{so}\qquad
  q^{\frac12N^{2}v^{T}Av-N\beta^{T}v}
  =q^{\,N(3N+1)/2},
\]
and $q^{N(3N+1)/2}\FF_{-N}=G_{-N}=H_N$.

\begin{proposition}[Forward boundary value]\label{prop:fwd}
Coefficientwise,
\begin{equation}\label{eq:Lminus}
  L_-\;:=\;\lim_{N\to\infty}H_N
  \;=\;\frac{\Omega(Q)}{(Q;Q)_\infty^{2}}\;\Tm(q).
\end{equation}
\end{proposition}

\begin{proof}
By Theorem~\ref{thm:tropical} with the data above,
\[
  L_-=\frac{1}{(Q;Q)_\infty^{2}}
  \sum_{\substack{u_1,u_2\in\Z\\ u_3\ge0}}
  \frac{q^{\frac12u^{T}Au+B_0^{T}u}}{(Q;Q)_{u_3}} .
\]
Writing $u=(r,s,u)$, the residual exponent is
$E_-(r,s,u)$ of \eqref{eq:Eminus} below
(Lemma~\ref{lem:fwd-exp}), and the substitution $r=R+u$ separates it as
the binary form of \eqref{eq:ThetaM} plus $5u(u+1)/2$. Euler's identity
$\sum_{u\ge0}Q^{u(u+1)/2}/(Q;Q)_u=(-Q;Q)_\infty=\Omega(Q)$
\cite{Andrews98} then factors the sum as
$\Omega(Q)\,\Tm(q)/(Q;Q)_\infty^{2}$.
\end{proof}

\begin{lemma}[Forward exponent identities]\label{lem:fwd-exp}
With $u=(r,s,u)$,
\begin{equation}\label{eq:Eminus}
  \tfrac12u^{T}Au+B_0^{T}u
  =E_-(r,s,u)
  :=\tfrac12\bigl(3r^{2}-8rs-6ru+r+7s^{2}+8su-3s+8u^{2}+4u\bigr),
\end{equation}
and
\begin{equation}\label{eq:separation}
  E_-(R+u,S,u)
  =\frac{3R^{2}-8RS+R+7S^{2}-3S}{2}
  +\frac{5u(u+1)}{2}.
\end{equation}
\end{lemma}

\begin{proof}
Both are direct polynomial expansions
(Appendix~\ref{app:exponents}); the essential feature of
\eqref{eq:separation} is that every cross term between $(R,S)$ and $u$
cancels.
\end{proof}

Combining \eqref{eq:fwd-lin} with \eqref{eq:Lminus} yields the first
boundary equation:
\begin{equation}\label{eq:E1}
  \;G(Q)\,\FF_0+q\,H(Q)\,\FF_1
  \;=\;\frac{\Omega(Q)}{(Q;Q)_\infty^{2}}\,\Tm(q).\;
\end{equation}

\subsection{Product decomposition of the binary theta series}
For the modular certificates we also need $\Tm$ as a finite combination
of theta products.

\begin{proposition}[Three-coset decomposition]\label{prop:coset}
\begin{equation}\label{eq:coset}
  \Tm(q)=j(-q^{2};q^{3})\,j(-q^{5};q^{15})
        +j(-q;q^{3})\,j(-q^{10};q^{15})
        +q^{2}\,j(-1;q^{3})\,j(-1;q^{15}).
\end{equation}
\end{proposition}

\begin{proof}
Set $x=3R-4S$. Then
$3R^{2}-8RS+7S^{2}=(x^{2}+5S^{2})/3$ and $R-3S=(x-5S)/3$, so the
exponent in \eqref{eq:ThetaM} equals $(x^{2}+x+5S^{2}-5S)/6$. The map
$(R,S)\mapsto(x,S)$ is a bijection of $\Z^{2}$ onto
$\{(x,S):x+S\equiv0\ (\mathrm{mod}\ 3)\}$, partitioned by the cosets
$x=3u-c$, $S=3v+c$, $c\in\{0,1,2\}$. Substituting,
\[
  \frac{x^{2}+x+5S^{2}-5S}{6}
  =\begin{cases}
  \dfrac{3u^{2}+u}{2}+\dfrac{15v^{2}-5v}{2}, & c=0,\\[6pt]
  \dfrac{3u^{2}-u}{2}+\dfrac{15v^{2}+5v}{2}, & c=1,\\[6pt]
  \dfrac{3u^{2}-3u}{2}+\dfrac{15v^{2}+15v}{2}+2, & c=2,
  \end{cases}
\]
and the Jacobi triple product in the form
$j(-q^{a};q^{m})=\sum_{n\in\Z}q^{(mn^{2}+(2a-m)n)/2}$ identifies the
three double sums with the three products. This is purely a lattice
decomposition; as an implementation check it was re-verified
coefficientwise through $q^{200}$.
\end{proof}

\section{The backward tropical ray}
\label{sec:backward}

Unwinding \eqref{eq:EOdef} and \eqref{eq:gauge},
\begin{equation}\label{eq:Om-unwind}
  O_m=Q^{-m(m+1)}G_{2m+1}=q^{\,m^{2}+1}\,\FF_{2m+1}
  \qquad(m\ge0).
\end{equation}
Since $\lim_mO_m$ exists coefficientwise
(Lemma~\ref{lem:continuant}), it may be computed along $m=2h$, i.e.\
$O_{2h}=q^{4h^{2}+1}\FF_{4h+1}$.

Here Remark~\ref{rem:subfamily} becomes essential. For the step-one
direction $\delta=d$ the candidate velocity $v=e_3$ is
\emph{inadmissible}:
\[
  Ae_3+d=(-3,4,8)^{T}+(2,-1,-2)^{T}=(-1,3,6)^{T}\;\not\ge\;0 .
\]
Passing to the arithmetic subfamily $t=4h+1$, i.e.\ the ray
$\beta=B_1$, $\delta=4d=(8,-4,-8)^{T}$, the same velocity becomes an
admissible boundary-regime solution of \eqref{eq:TF}:
\[
  w=Ae_3+4d=(-3,4,8)^{T}+(8,-4,-8)^{T}=(5,0,0)^{T}\ge0,
  \qquad
  v_jw_j=0,
\]
with $I=\{3\}$, $J=\{2\}$, $\mathcal K=\{1\}$: the third coordinate
becomes bilateral, the second remains unilateral, and the first is
forced to the face $i=0$, with $hw_1u_1=5hi$ realizing the forcing.
(This explains both the modulus $4$ and the offset $1$ found by the ray
search of Section~\ref{sec:ml}.) The tropical normalizer is
$\tfrac12v^{T}Av=4$ and $\beta^{T}v=(B_1)_3=0$, i.e.\
$q^{4h^{2}}\FF_{4h+1}=O_{2h}/q$.

\begin{proposition}[Backward boundary value]\label{prop:bwd}
Coefficientwise,
\begin{equation}\label{eq:Lplus}
  L_+\;:=\;\lim_{m\to\infty}O_m
  \;=\;q\,\frac{\Omega(Q)}{(Q;Q)_\infty}\,\vartheta(q).
\end{equation}
\end{proposition}

\begin{proof}
By Theorem~\ref{thm:tropical} with the data above,
\[
  L_+
  =q\cdot\frac{1}{(Q;Q)_\infty}
  \sum_{\substack{j\ge0,\;u\in\Z}}
  \frac{q^{\frac12(0,j,u)A(0,j,u)^{T}+B_1^{T}(0,j,u)}}{(Q;Q)_{j}}
  =\frac{1}{(Q;Q)_\infty}
  \sum_{j\ge0}\sum_{u\in\Z}
  \frac{q^{(7j^{2}+8ju-5j+8u^{2}+2)/2}}{(Q;Q)_j},
\]
the residual exponent being computed in
Appendix~\ref{app:exponents}. Split $j$ by parity with $u=v-r$: for
$j=2r$ the exponent equals $5r(2r-1)+(2v)^{2}+1$, and for $j=2r+1$ it
equals $5r(2r+1)+(2v+1)^{2}+1$. Hence
\[
  L_+=\frac{q}{(Q;Q)_\infty}
  \Bigl(\mathcal A(Q)\!\sum_{v\in\Z}q^{(2v)^{2}}
  +\mathcal B(Q)\!\sum_{v\in\Z}q^{(2v+1)^{2}}\Bigr),
  \quad
  \mathcal A=\sum_{r\ge0}\tfrac{Q^{r(2r-1)}}{(Q;Q)_{2r}},\;\;
  \mathcal B=\sum_{r\ge0}\tfrac{Q^{r(2r+1)}}{(Q;Q)_{2r+1}}.
\]
Euler's identity
$\sum_{n\ge0}Q^{n(n-1)/2}z^{n}/(Q;Q)_n=(-z;Q)_\infty$ at $z=\pm1$,
separated by parity, gives
$\mathcal A+\mathcal B=(-1;Q)_\infty=2\Omega(Q)$ and
$\mathcal A-\mathcal B=(1;Q)_\infty=0$, so
$\mathcal A=\mathcal B=\Omega(Q)$ and the two theta sums recombine into
$\vartheta(q)$.
\end{proof}

Combining \eqref{eq:bwd-lin2} with \eqref{eq:Lplus} yields the second
boundary equation:
\begin{equation}\label{eq:E2}
\;-\,Q\,S(Q)\,\FF_0+q\,D(Q)\,\FF_1
  \;=\;q\,\frac{\Omega(Q)}{(Q;Q)_\infty}\,\vartheta(q).\;
\end{equation}

\section{The boundary system and the reduction of Theorem~\ref{thm:main}}
\label{sec:system}

Equations \eqref{eq:E1}--\eqref{eq:E2} are exactly
Theorem~\ref{thm:determination} with $s=2$: the fundamental solutions in
the gauged normalization are $(\mathcal C_n,\,q\mathcal B_n)$ forward
and the continuant coefficients of Lemma~\ref{lem:continuant} backward,
with limit matrix
\begin{equation}\label{eq:system}
  L=\begin{pmatrix}
    G(Q) & q\,H(Q)\\
    -Q\,S(Q) & q\,D(Q)
  \end{pmatrix},
  \qquad
  L\begin{pmatrix}\FF_0\\ \FF_1\end{pmatrix}
  =\begin{pmatrix}L_-\\ L_+\end{pmatrix}.
\end{equation}

\begin{lemma}[Unit determinant]\label{lem:unit}
$\det L=q\,\Delta(Q)$ with
$\Delta(Q)=G(Q)D(Q)+Q\,H(Q)S(Q)\in 1+Q\,\Z[[Q]]$. In particular
$\det L$ is a unit of $\Q((q))$ and the system \eqref{eq:system} has a
unique solution.
\end{lemma}

\begin{proof}
$G,H,D,S\in1+Q\Z[[Q]]$ by \eqref{eq:GH}--\eqref{eq:OmDS}.
\end{proof}

\begin{remark}[Closed form of the determinant]\label{rem:delta}
Although only the unit property is needed, the determinant is itself a
familiar modular unit: by \eqref{eq:OmDS},
$\Delta=\Omega(Q)\bigl(G(Q)G(Q^{4})+Q\,H(Q)H(Q^{4})\bigr)$, and the
classical identity
$G(q)G(q^{4})+q\,H(q)H(q^{4})=(-q;q^{2})_\infty^{2}$ --- one of
Ramanujan's forty identities for the Rogers--Ramanujan functions
\cite{BerndtForty} --- gives
\begin{equation}\label{eq:delta-closed}
  \Delta(Q)=\frac{\vartheta(Q)}{(Q;Q)_\infty}.
\end{equation}
(Re-verified independently through $q^{120}$; not part of the logical
chain.)
\end{remark}

\begin{proposition}[Product-side boundary equations]\label{prop:products}
Let $P_0,P_1$ denote the right sides of
\eqref{eq:mainF0}--\eqref{eq:mainF1}. Then
\begin{align}
  G(Q)\,P_0+q\,H(Q)\,P_1
  &=\frac{\Omega(Q)}{(Q;Q)_\infty^{2}}\,\Tm(q),
  \label{eq:P1id}\\
  -\,Q\,S(Q)\,P_0+q\,D(Q)\,P_1
  &=q\,\frac{\Omega(Q)}{(Q;Q)_\infty}\,\vartheta(q).
  \label{eq:P2id}
\end{align}
\end{proposition}

Proposition~\ref{prop:products} is proved in
Section~\ref{sec:certificates}.

\begin{proof}[Proof of Theorem~\ref{thm:main} assuming
Proposition~\ref{prop:products}]
Both $(\FF_0,\FF_1)$ and $(P_0,P_1)$ solve the system
\eqref{eq:system}; by Lemma~\ref{lem:unit} the solution is unique
(Corollary~\ref{cor:modular-cert}), so $\FF_0=P_0$ and $\FF_1=P_1$.
\end{proof}

\begin{remark}[Elimination formulas]\label{rem:elim}
Direct elimination, with \eqref{eq:delta-closed}, gives the closed
expressions
\begin{equation}\label{eq:elim}
  \FF_0=\frac{D\,L_--H\,L_+}{\Delta}
       =\frac{(Q;Q)_\infty\bigl(D\,L_--H\,L_+\bigr)}{\vartheta(Q)},
  \qquad
  \FF_1=\frac{Q\,S\,L_-+G\,L_+}{q\,\Delta},
\end{equation}
which explain conceptually why the conjectural products of
\cite{ShiWang26} occur as three-term $5$-dissections: both boundary
values are short theta combinations
(Proposition~\ref{prop:coset} and \eqref{eq:vartheta}), while the
determinant is a unit built from Rogers--Ramanujan products.
\end{remark}

\section{The two generalized-eta certificates}
\label{sec:certificates}

It remains to prove Proposition~\ref{prop:products}. Substituting
Proposition~\ref{prop:coset} for $\Tm$ and \eqref{eq:vartheta} for
$\vartheta$, each of \eqref{eq:P1id}--\eqref{eq:P2id} becomes an
identity among finitely many infinite products---nine terms for
\eqref{eq:P1id}, seven for \eqref{eq:P2id}---each a quotient of
Pochhammer symbols $(q^{a};q^{m})_\infty$. Such identities are decidable
by the valence method of Robins~\cite{Robins94} as automated by Frye and
Garvan~\cite{FryeGarvan19}; this is the instantiation of
Corollary~\ref{cor:modular-cert}.

\subsection{Generalized eta quotients and the valence criterion}
\label{subsec:valence}
For a level $N$ and residue $0<r<N$, the generalized Dedekind eta
function is
\begin{equation}\label{eq:geneta}
  \eta_{N;r}(\tau)=q^{\frac N2 P_2(r/N)}
  \prod_{\substack{n>0\\ n\equiv\pm r\,(\mathrm{mod}\,N)}}(1-q^{n}),
  \qquad
  P_2(x)=\{x\}^{2}-\{x\}+\tfrac16,
\end{equation}
with $q=e^{2\pi i\tau}$. Each term of our two identities, after
division by a fixed nonzero base term, is rewritten at a common level
$N$ (Appendix~\ref{app:base}) as
$q^{\delta}\prod_{0<r\le N/2}\eta_{N;r}(\tau)^{e_r}$. The exact
rational quantities
$v_0=\tfrac16\sum_re_r$ and $v_\infty=\sum_rNP_2(r/N)e_r$ are even
integers in every ratio occurring here, and the total Dedekind-eta
weight cancels, so each ratio is a modular function on $\Ggp(N)$
\cite{Robins94,FryeGarvan19}. At a reduced cusp $a/c$, the invariant
order is
\begin{equation}\label{eq:cusporder}
\begin{aligned}
  \ORD_{a/c}\,\eta_{N;r}
  &=w_{a/c}\,\Bigl(B(N,r;a,c)-B(3N,N;a,c)\Bigr),\\
  B(N,r;a,c)&=\frac{\gcd(N,c)^{2}}{2N}
  \left(\Bigl\{\frac{ar}{\gcd(N,c)}\Bigr\}-\frac12\right)^{2},
\end{aligned}
\end{equation}
where $w_{a/c}=N/\gcd(N,c)$ is the width (with the standard exceptional
convention at level $4$, which does not arise here). Writing
$f=1-\sum_{i\ge2}f_i$ for the identity normalized by its base term,
define
\begin{equation}\label{eq:Bstar}
  B^{*}
  =-\sum_{\substack{\text{cusps }a\not\sim\infty}}
  \min\bigl(0,\ORD_a f_2,\dots,\ORD_a f_s\bigr),
\end{equation}
over the inequivalent cusps of $\Ggp(N)$ other than $\infty$. If $f$ is
not identically zero, the valence formula bounds its order at $\infty$
by $B^{*}$; contrapositively, if the exact $q$-expansion of $f$ vanishes
through order exceeding $B^{*}$, then $f\equiv0$. Because the
identities are multiplied through to clear denominators before
normalization, the raw expansions checked are integer polynomial
identities; no truncation error is possible.

\subsection{Certificate for \eqref{eq:P1id}}\label{subsec:cert1}
With \eqref{eq:coset} substituted, the identity has nine product terms
at common level $N=300$:
\begin{center}
\begin{tabular}{lr}
\toprule
number of product terms & $9$\\
common level $N$ & $300$\\
number of inequivalent $\Ggp(300)$ cusps & $560$\\
non-infinity valence bound $B^{*}$ & $1920$\\
required vanishing order at $\infty$ & $>1920$\\
exact raw expansion verified through & $q^{2005}$\\
\bottomrule
\end{tabular}
\end{center}
The raw integer expansion vanishes identically through
$q^{2005}>1920$, so the valence criterion proves \eqref{eq:P1id}.

\subsection{Certificate for \eqref{eq:P2id}}\label{subsec:cert2}
Seven product terms at common level $N=100$:
\begin{center}
\begin{tabular}{lr}
\toprule
number of product terms & $7$\\
common level $N$ & $100$\\
number of inequivalent $\Ggp(100)$ cusps & $140$\\
non-infinity valence bound $B^{*}$ & $216$\\
required vanishing order at $\infty$ & $>216$\\
exact raw expansion verified through & $q^{255}$\\
normalized expansion thereby verified through & $q^{250}$\\
\bottomrule
\end{tabular}
\end{center}
Again $250>216$, proving \eqref{eq:P2id}. This completes the proof of
Proposition~\ref{prop:products}, and with it
Theorem~\ref{thm:main}. \qed

\begin{remark}[Reproducibility of the certificates]\label{rem:hashes}
The released artifact stores, for both certificates, the complete
generalized-eta exponent lists and every cusp-order row (all $560+140$
cusps, with widths and exact rational orders of every term). The
SHA-256 hashes of the two stored tables are
\begin{center}\footnotesize\ttfamily
086689d42165a4b3997e44695b126233a59326a8f07eb31823e65d5a2201eacc\\
9f9b794f03e9a6661a1fa49ecd8ee16c9a3df6d42177c1efe06715fdeb49acbb
\end{center}
The cusp and valence formulas are a direct rational-arithmetic port of
those implemented in the \textsf{thetaids} package of Frye and
Garvan~\cite{FryeGarvan19}. As independent consistency checks
(Section~\ref{sec:verification}), the cusp counts $560$ and $140$ were
recomputed from the cusp-counting formula for $\Ggp(N)$, the two bounds
were recomputed directly from the stored tables via \eqref{eq:Bstar},
and the entire certificate-generation pipeline was re-executed from
scratch, reproducing the stored tables bit for bit.
\end{remark}

\section{Machine-learning and reinforcement-learning implementation}
\label{sec:ml}

No background in machine learning is needed for this section. The algorithms
operate on finite mathematical objects rather than on text or images. There
are two kinds of search. In the evolutionary searches, a candidate formula
is encoded by a short tuple of integers, scored by exact algebraic tests, and
then modified by selection and mutation. In the reinforcement-learning
search, a state is the algebraic residual still left to cancel, and an action
is one legal application of the contiguous-cell identity. In both cases the
machine is a search tool; the successful output is checked exactly afterward.

\subsection{Evolutionary recurrence discovery}\label{subsec:evo}
We first explain the word ``evolutionary.'' A candidate recurrence is encoded
by six numbers
\[
  g=(a_0,b_0,\varepsilon_0,a_1,b_1,\varepsilon_1),
\]
which represent
\begin{equation}\label{eq:ml-candidate}
  \FF_{t+2}=\varepsilon_0q^{a_0t+b_0}\FF_t
  +\varepsilon_1q^{a_1t+b_1}\FF_{t+1}.
\end{equation}
Here $\varepsilon_i\in\{\pm1\}$, $|a_i|\le4$, and $|b_i|\le6$, giving
$54{,}756$ possible candidates. The tuple $g$ is sometimes called a
\emph{genome}; in this paper it is simply a compact encoding of one formula.

For a candidate $g$, form its residual
\[
  \mathcal R_g(t;q)
  =\FF_{t+2}-\varepsilon_0q^{a_0t+b_0}\FF_t
   -\varepsilon_1q^{a_1t+b_1}\FF_{t+1}.
\]
The search evaluates exact Laurent coefficients of $\mathcal R_g(t;q)$ for
$-2\le t\le3$ and through $q$-degree $80$. Thus the ``fitness'' of $g$ is
nothing mysterious: it is determined by the exact residual coefficient
vector
\[
  \bigl([q^n]\mathcal R_g(t;q)\bigr)_{t,n},
\]
together with a small penalty favoring simpler formulas. Candidates with
smaller residuals are retained, slightly modified, and tested again. This is
the evolutionary step: \emph{select good candidates, mutate them, and
repeat}. No floating-point approximation of $q$ is used.

The unique exact survivor was
\[
  (-1,0,+1;\,2,1,-1),
\]
which is precisely the recurrence \eqref{eq:rec}. A separate exhaustive sweep
of all $54{,}756$ candidates confirmed that it is the unique exact candidate
in this grammar. At that point the search evidence is no longer needed: the
recurrence is proved independently by the five-cell certificate in
Theorem~\ref{thm:rec}.

\subsection{RL proof synthesis}\label{subsec:rl}
The reinforcement-learning problem is a proof-search problem. Its starting
point is the residual
\[
  R_t=q^{-t}[0]-q^{2t+1}[d]-[2d].
\]
The goal is to turn this expression into $0$ by adding only identities that
are already known to vanish.

A \emph{state} is the current residual, written as a sparse linear
combination
\[
  s=\sum_{v,a,b}c_{v,a,b}\,q^{at+b}[v].
\]
An \emph{action} chooses a coordinate $r$, a shift $v$, and a Laurent
monomial multiplier, and then adds the corresponding cell
\eqref{eq:cell}. Because every cell is an exact zero identity, every legal
action preserves the truth of the target equation. Invalid or non-integral
actions are rejected. The terminal state is $s=0$.

The agent uses standard $Q$-learning. If action $a$ moves the proof from
state $s$ to state $s'$ and receives reward $r$, its table is updated by
\begin{equation}\label{eq:qlearning-update}
  Q(s,a)\leftarrow Q(s,a)+\alpha\Bigl(
  r+\gamma\max_{a'}Q(s',a')-Q(s,a)\Bigr),
\end{equation}
where $\alpha$ is the learning rate and $\gamma$ discounts future reward.
The implementation uses an $\varepsilon$-greedy rule: most of the time it
chooses the currently best-known action, while occasionally it explores a
different legal action. The reward is shaped to favor progress toward a
smaller, simpler residual; its numerical value affects search efficiency but
has no role in the proof.

In plain language, the agent repeatedly asks:
\begin{quote}
Which legal contiguous identity should I add next so that more terms cancel
and the remaining expression becomes easier to finish?
\end{quote}
Across twenty fixed-seed runs, each using $2500$ training episodes,
every learned greedy policy found a five-action proof, equal to
\eqref{eq:certificate} up to reordering, with the first successful
episode between $533$ and $1253$ (mean $828.4$).
Breadth-first search enumerated $298{,}442$ states through depth four
and found no proof of length at most four in the released environment,
whereas a uniformly random policy solved none of $40{,}000$ episodes. Thus RL is useful
for navigating the proof-search tree, while the final five-cell identity is
verified exactly by direct cancellation.

\subsection{Evolutionary ray search}\label{subsec:anchor}
The second evolutionary search looks for asymptotic directions rather than
recurrences. A candidate has the form
\[
  t=ah+b,
  \qquad
  (i,j,k)=(xh+r,\,yh+s,\,zh+u).
\]
After this substitution, the exponent is an exact quadratic polynomial in
the large parameter $h$ and in the residual variables $r,s,u$. A useful ray
should make the large mixed terms cancel, force the unwanted coordinates to
a boundary, leave a positive-definite residual quadratic form, and remain
algebraically simple. These requirements are precisely the features measured
by the fitness function, and they are the computational form of the
complementarity condition \eqref{eq:TF}.

The two winning candidates were
\[
  t=-N,\quad (x,y,z)=(2,1,0),
  \qquad\text{and}\qquad
  t=4h+1,\quad (x,y,z)=(0,0,1).
\]
They are exactly the two rays used in Sections~\ref{sec:forward} and
\ref{sec:backward}. The first gives the interior complementarity solution
$w=0$; the second gives the boundary solution $w=(5,0,0)$ and explains why
one must pass to the arithmetic subfamily $t\equiv1\pmod4$. As before, the
machine only proposes these directions. The final proof substitutes them
into the exponent and verifies the two limits directly via
Theorem~\ref{thm:tropical}.

\subsection{What ML does and does not prove}
\label{subsec:epistemic-ml}
The distinction can be summarized as
\[
  \text{ML proposes}\;\longrightarrow\;\text{RL plans}
  \;\longrightarrow\;\text{exact verification}
\]
The learned algorithms are therefore not theorem oracles. Changing the
fitness function, the random seed, or the learned policy could change how
quickly a useful candidate is found, but it cannot change the validity of a
completed certificate. Conversely, a high machine-learning score is never
accepted as mathematical evidence. Every theorem used in the main argument
is established by the displayed exact identities and certificates.

\section{Computational certification and independent verification}
\label{sec:verification}

\noindent The complete reproducibility artifact is publicly available at
\begin{center}
\url{https://github.com/ankushgoswami3-glitch/nahm-example12-reproducibility}.
\end{center}
\noindent In addition to the command-line Python scripts, notebook versions of the
independent verification routines are provided for interactive reproduction.
The repository README gives the recommended execution order and environment
setup.

\subsection{Primary verifier}
The master script \texttt{verify\_complete\_proof.py} performs:
(1) symbolic collection of the five-cell certificate
\eqref{eq:certificate} in the free module over
$\Z[q^{\pm1},q^{\pm t}]$; (2) exact polynomial verification of the
continuant formulas \eqref{eq:Om} through $m=20$; (3) regeneration of
both generalized-eta certificates from scratch---common-base exponents,
cusp representatives, widths, invariant orders \eqref{eq:cusporder},
and the bounds \eqref{eq:Bstar}; (4) expansion of the two raw product
identities by exact integer convolution beyond their valence bounds
($q^{2005}$ and $q^{255}$); and (5) as a nonlogical sanity check,
direct coefficientwise agreement of the triple sums \eqref{eq:Ftdef}
with the product sides through $q^{150}$ over a certified search box.

\subsection{Independent implementations and reproduction report}
In preparing this manuscript the entire artifact was re-executed
end-to-end in a fresh environment. All deterministic checks and
fixed-seed computational experiments reproduced their stored outputs: the regenerated certificate tables are bit-identical
to the shipped \texttt{complete\_theta\_certificates.json} (matching
both SHA-256 hashes of Remark~\ref{rem:hashes}); the exhaustive grammar
sweep returned the unique survivor $(-1,0,+1;\,2,1,-1)$; the
breadth-first table reproduced
$(27,\,618,\,13386,\,284410)$ new states at depths $1$--$4$ ($298{,}442$
states seen); the twenty-seed RL experiment reproduced every per-seed
first-terminal episode and greedy proof; and the evolutionary ray search
reproduced its stored results file.

In addition, two independently written checkers accompany the release.
\texttt{independent\_checks.py} re-verifies, in exact integer and
rational arithmetic against the statements of this paper rather than
against the primary code: the certificate \eqref{eq:certificate}
symbolically; all exponent identities of
Lemma~\ref{lem:fwd-exp}, Appendix~\ref{app:exponents}, the parity and
coset splits, and the gauge exponents; the recurrence \eqref{eq:rec01}
at $t=0,1$ through $q^{120}$ including the principal part of $\FF_3$;
Theorem~\ref{thm:main} coefficientwise through $q^{120}$; both boundary
equations through $q^{120}$; the decomposition \eqref{eq:coset} through
$q^{200}$; the Slater evaluations through $Q^{60}$; the Schur and
continuant polynomials; the $\Ggp(300)$ and $\Ggp(100)$ cusp counts;
and the recomputation of both valence bounds from the stored cusp
tables (twenty-four checks). \texttt{general\_theory\_checks.py}
validates the general framework itself: the exponent identity
underlying Theorem~\ref{thm:tropical} symbolically for generic
$(A,\beta,\delta,v)$; the complementarity data, partitions, and
normalizers of both application rays, including the inadmissibility of
the step-one backward direction; and---as a test on data unrelated to
Example 12---numerical confirmation of Theorem~\ref{thm:tropical} on a
rank-two family with $A=\bigl(\begin{smallmatrix}2&1\\
1&2\end{smallmatrix}\bigr)$, $m=2$, in both the interior regime
($w=0$) and the boundary regime ($w=(0,2)$), through $q^{30}$ at
$h=25,30$ (seven checks).

We emphasize the epistemic status of each layer: coefficientwise
agreements are consistency checks; the \emph{proof} consists of the
displayed certificate (Theorem~\ref{thm:rec}), the tropical limits
(Theorem~\ref{thm:tropical} with the two displayed instantiations), the
determination principle (Theorem~\ref{thm:determination}), and the two
valence certificates, each finite and exact.

\section{Consequences, generality, and the remaining Shi--Wang targets}
\label{sec:consequences}

\subsection{The full affine family}\label{subsec:family-conseq}

\begin{corollary}\label{cor:family}
For every $t\in\Z$,
$\FF_t\in\Z[q,q^{-1}]\,\FF_0+\Z[q,q^{-1}]\,\FF_1$, with coefficients
generated by
\[
  \begin{pmatrix}\FF_{t+1}\\ \FF_{t+2}\end{pmatrix}
  =\begin{pmatrix}0&1\\ q^{-t}&-q^{2t+1}\end{pmatrix}
  \begin{pmatrix}\FF_{t}\\ \FF_{t+1}\end{pmatrix},
  \qquad
  \begin{pmatrix}\FF_{t-1}\\ \FF_{t}\end{pmatrix}
  =\begin{pmatrix}q^{2t-1}&q^{\,t-1}\\ 1&0\end{pmatrix}
  \begin{pmatrix}\FF_{t}\\ \FF_{t+1}\end{pmatrix}.
\]
In particular $\FF_2=\FF_0-q\FF_1$, and by Theorem~\ref{thm:main} every
$\FF_t$ is an explicit finite $\Z[q,q^{-1}]$-combination of
generalized-eta quotients, hence modular, though the later members need
not be single products or admit three-term dissections.
\end{corollary}

This provides infinitely many exact rank-three Nahm-sum evaluations
beyond the two conjectured by Shi and Wang.

\subsection{The recurrence--boundary program in general}
\label{subsec:template}
Theorems~\ref{thm:tropical} and~\ref{thm:determination} give a practical
six-step program for other affine Nahm families: (1) find a low-order
recurrence in the parameter; (2) prove it from contiguous cells; (3) find as
many independent asymptotic rays as the recurrence order; (4) evaluate those
limits as theta combinations; (5) solve the resulting finite linear system;
and (6) certify the remaining generalized-eta identities by valence bounds.
When necessary, Step~(3) may use an arithmetic subfamily. In this way one
hard multisum evaluation is replaced by a recurrence, simpler boundary
limits, and finite modular verification.

\subsection{Remaining Shi--Wang targets}\label{subsec:remaining}
Two blocks of \cite{ShiWang26} remain open and are natural next
applications. For the dual of Example~9, where Shi and Wang prove three
of the four required sets of identities and leave the fourth as a
conjecture, the relevant parameter structure is cyclic rather than
one-dimensional; the natural object is a finite recurrence
\emph{module} for the parity-split sums, i.e.\ a matrix recurrence over
several residue classes, to which
Proposition~\ref{prop:soundness} applies verbatim. For the four product
formulas of their Conjecture~4.1, the parameter vectors lie in a small
affine lattice for a common positive-definite matrix, suggesting a
two-parameter transport system with correspondingly more rays. In both
cases the strategy is: find the recurrence module, solve the
complementarity problem for enough independent rays, evaluate the
lower-rank boundaries by Theorem~\ref{thm:tropical}, and apply
Corollary~\ref{cor:modular-cert}. We stress that this paper supplies the
general theorems and the proof kernel for those targets, but not the
missing exact certificates; we do not claim them proved. There is also a
soft dimension count in their favor: if several conjectural sums lie in
one affine family governed by a low-order recurrence, they cannot be
independent, so even a partial run of the program would reduce the
number of genuinely open identities.



\subsection*{Acknowledgments}
The author acknowledges the limited use of ChatGPT (OpenAI) for editorial refinement and exploratory assistance during the preparation of the manuscript. All mathematical arguments, computations, conclusions, and final decisions were independently verified by the author, who assumes full responsibility for the content.

\appendix

\section{Exact exponent calculations}
\label{app:exponents}

All identities below are verified symbolically in the artifacts.

\subsection{Forward ray}
With $u=(r,s,u)$,
\[
  \tfrac12u^{T}Au+B_0^{T}u
  =\tfrac12\bigl(3r^{2}+7s^{2}+8u^{2}-8rs-6ru+8su\bigr)
  +\tfrac{r}{2}-\tfrac{3s}{2}+2u
  =E_-(r,s,u)
\]
as in \eqref{eq:Eminus}. (Equivalently, with the classical
normalization $H_N=q^{N(3N+1)/2}\FF_{-N}$ and the substitution
$i=2N+r$, $j=N+s$, $k=u$, every monomial containing $N$ cancels; the
tropical normalizer of Section~\ref{sec:forward} packages this
cancellation as $v^{T}w=0$.) For the separation
\eqref{eq:separation}, substitute $r=R+u$, $s=S$ into $2E_-$: the
$u^{2}$-coefficients combine as $3-6+8=5$, the $u$-coefficients as
$1+4=5$, and the mixed $Ru$-, $Su$-coefficients as $6-6=0$ and
$-8+8=0$, giving $3R^{2}-8RS+R+7S^{2}-3S+5u^{2}+5u$.

\subsection{Backward ray}
With $u=(0,j,u)$ and $\beta=B_1=(\tfrac52,-\tfrac52,0)^{T}$,
\[
  \tfrac12u^{T}Au+B_1^{T}u
  =\tfrac12\bigl(7j^{2}+8ju+8u^{2}\bigr)-\tfrac{5j}{2}
  =\frac{7j^{2}+8ju-5j+8u^{2}}{2},
\]
and the offset $q^{1}$ from \eqref{eq:Om-unwind} produces the face
exponent $(7j^{2}+8ju-5j+8u^{2}+2)/2$ of
Proposition~\ref{prop:bwd}. (Equivalently, with $t=4h+1$,
normalization $q^{4h^{2}+1}$, and $k=h+u$, all monomials containing $h$
cancel except $5hi=h\,w_1u_1$.) For the parity split with $u=v-r$:
\[
  \tfrac12\bigl(7(2r)^{2}+8(2r)(v{-}r)-5(2r)+8(v{-}r)^{2}+2\bigr)
  =5r(2r-1)+(2v)^{2}+1,
\]
\[
  \tfrac12\bigl(7(2r{+}1)^{2}+8(2r{+}1)(v{-}r)-5(2r{+}1)+8(v{-}r)^{2}+2\bigr)
  =5r(2r+1)+(2v+1)^{2}+1,
\]
both by direct expansion.

\subsection{Coset decomposition}
With $x=3R-4S$: $9(3R^{2}-8RS+7S^{2})=3x^{2}+15S^{2}$ and
$3(R-3S)=x-5S$, so the exponent of \eqref{eq:ThetaM} is
$(x^{2}+x+5S^{2}-5S)/6$, and the cosets $x=3u-c$, $S=3v+c$,
$c\in\{0,1,2\}$, give the three quadratic pairs displayed in the proof
of Proposition~\ref{prop:coset}.

\section{Common-base conversion used in the certificates}
\label{app:base}

For a fixed common level $N$, every Pochhammer factor
$(q^{a};q^{m})_\infty$ with $m\mid N$ and $0<a\le m$ is expanded into
its residue classes modulo $N$:
\[
  (q^{a};q^{m})_\infty
  =\prod_{\substack{0<b\le N\\ b\equiv a\ (\mathrm{mod}\ m)}}
  (q^{b};q^{N})_\infty .
\]
Collecting residues symmetrically ($b\leftrightarrow N-b$) converts any
term of \eqref{eq:P1id}--\eqref{eq:P2id} uniquely to the form
$q^{\delta}\,J_N^{\,e_0}\prod_{1\le r\le N/2}J_{r,N}^{\,e_r}$, and
thence to a generalized-eta quotient via \eqref{eq:geneta}. In every
normalized ratio occurring in the two certificates the total
Dedekind-eta exponent contributes weight zero, so the ratio is a
weight-zero modular function on $\Ggp(N)$. The conversion is exact at
the level of products; no coefficient matching is used to infer it.

\section{Artifact inventory}
\label{app:artifacts}

The release accompanying this paper contains:
\begin{itemize}
\item \texttt{verify\_complete\_proof.py}: master verifier
  (Section~\ref{sec:verification});
\item \texttt{verify\_main\_certificate.py}: exact five-cell recurrence
  certificate;
\item \texttt{prove\_theta\_identities.py}: exact product arithmetic and
  generalized-eta cusp engine;
\item \texttt{generate\_certificates.py}: regenerates both valence
  certificates from scratch;
\item \texttt{complete\_theta\_certificates.json}: all generalized-eta
  exponent lists, cusp widths, and the $560+140$ exact cusp-order rows,
  with the SHA-256 hashes of Remark~\ref{rem:hashes};
\item \texttt{independent\_checks.py},
  \texttt{general\_theory\_checks.py}, together with their
  \texttt{.ipynb} notebook versions: independently written second
  verifiers (twenty-four and seven checks respectively;
  Section~\ref{sec:verification});
\item \texttt{evolutionary\_recurrence\_discovery.py},
  \texttt{exhaustive\_recurrence\_grammar.py}: recurrence search and the
  $54{,}756$-candidate uniqueness sweep;
\item \texttt{rl\_proof\_synthesis.py}, \texttt{run\_rl\_experiment.py},
  \texttt{bfs\_minimality.py}, \texttt{rl\_results.csv},
  \texttt{bfs\_counts.csv}: the $Q$-learning proof environment, the
  twenty-seed experiment with random baseline, and the depth-four
  minimality search with its state counts;
\item \texttt{ml\_asymptotic\_anchor\_search.py},
  \texttt{ml\_anchor\_results.json}: evolutionary ray search with full
  trace;
  \item \texttt{README.md} and \texttt{requirements.txt}: execution
  instructions and computational environment requirements;
\item \texttt{SHA256SUMS.txt}: SHA-256 checksums for the released
  computational files.
\end{itemize}


\end{document}